\documentclass[11pt]{article}

\usepackage{amssymb}
\usepackage{graphicx}
\usepackage{hyperref}
\usepackage{graphicx,type1cm,eso-pic,color}

\usepackage{verbatim}

\usepackage{amsmath,amsthm,amssymb}
\usepackage{bm,color}
\usepackage{xparse}          

\newcommand{\pr}[1]{\mathrm{P}\left(#1\right)}

\newcommand{\grph}{\mathcal{G}}
\newcommand{\nus}{\nu\nu}
\newcommand{\gk}{\kappa}
\newcommand{\gks}{\gk\gk}
\newcommand{\bA}{\bm{A}}

\newcommand{\bJ}{\bm{E}}

\newcommand{\bF}{\bm{F}}
\newcommand{\bI}{\bm{I}}
\newcommand{\bG}{\bm{G}}
\newcommand{\bP}{\bm{P}}
\newcommand{\bN}{\bm{N}}
\newcommand{\bM}{\bm{M}}

\newcommand{\bz}{\bm{0}}

\newcommand{\bPsi}{\bm{\Psi}}
\newcommand{\bPhi}{\bm{\Phi}}

\newcommand{\Rp}{\mathbb{R}_{_+}}
\newcommand{\stspc}{\mathcal{S}}
\newcommand{\tofrom}{\leftrightarrow}

\def\bc{\begin{center}}
\def\ec{\end{center}}
\def\noi{\noindent}

\DeclareDocumentCommand{\set}{m o}
  {\left\{#1\IfNoValueTF{#2}{}{\,:\,#2}\right\}}
\DeclareDocumentCommand{\pr}{m o}{\mathrm{P}\!
  \IfNoValueTF{#2}{}{\left.}\left\{#1\IfNoValueTF{#2}{}{\,\right|#2}\right\}}

\newtheorem{thm}{Theorem}[section]
\newtheorem{prop}{Proposition}[section]
\newtheorem{theorem}{Theorem}[section]
\newtheorem{lemma}[thm]{Lemma}
\newtheorem{defn}[theorem]{Definition}

\begin{document}
\setcounter{page}{1}
\newtheorem{t1}{Theorem}[section]
\newtheorem{d1}{Definition}[section]
\newtheorem{c1}{Corollary}[section]
\newtheorem{l1}{Lemma}[section]
\newtheorem{r1}{Remark}[section]



\thispagestyle{empty}
\markboth{}{}

\pagestyle{myheadings}

\date{}



\vspace{.1in} 

{\baselineskip 20truept

\bc
{\Large {\bf First Passage Moments of Finite-State Semi-Markov Processes}} 
\ec

\vspace{.1in} 

\bc
{\large {\bf Richard L. Warr}} \\
{\large {\it Brigham Young University}} \\
{\large {\bf James D. Cordeiro}}\\
{\large {\it University of Dayton}}
\ec 
}

\vspace{.1in}
\baselineskip 18truept


\begin{abstract}
In this paper, we discuss the computation of first-passage moments of 
a time-homogeneous semi-Markov process (SMP) with finite state space 
to certain of its states that possess the property of universal accessibility 
(UA). A UA state is one which is accessible from any other state of the SMP, 
but which may or may not connect back to one or more other states. An important 
characteristic of UA is that it is the state-level version of the oft-invoked 
process-level property of irreducibility. We adapt existing results for 
irreducible SMPs to the derivation of an analytical matrix expression for the 
first passage moments to a single UA state of the SMP. In addition, consistent 
estimators for these first passage moments are given.
\end{abstract}



\section{Introduction}
The first passage distribution of a discrete-state stochastic model such
as a Markov or semi-Markov process is a fundamental quantity of interest.
Results for the first passage distribution of a regular Markov process, which is
commonly defined as irreducible and aperiodic, are detailed in the classical
monograph by Kemeny and Snell \cite{kemeny_snell} and were subsequently 
developed by Hunter \cite{hunter2007simple,hunter2018}. Nevertheless, many 
multi-state models that arise in practice use reducible models. One
of the early examples of a SMP model that incorporates transient and absorbing
states comes from clinical management science \cite{weiss1965semi}. Recent
applications from reliability and survival analysis that require transient and
absorbing states include a nuclear pipeline rupture model of Veeramany, et al.
\cite{veeramany2011reliability} 
and the reliability-risk model for credit score
ratings of D'Amico, et al. \cite{DAmico2006}. In such applications,
absorbing states are commonly used to indicate failure while
transient states arise from degradation levels that may only be visited once.


Modeling applicability of multistate jump models may be further extended
by considering the semi-Markov process (SMP), which generalizes
the continuous-time Markov chain. 
Since the seminal works of Levy 
\cite{levy_2_54,levy_1_54} and Smith \cite{smith_55}, semi-Markov processes 
(SMPs) have been utilized as a framework for a wide variety of applications 
within the scientific literature. Much of the interest is due to the fact that 
the SMP relaxes the assumption of exponential sojourn times, which is not 
always appropriate, while retaining the memoryless characteristic of a Markov 
chain at the transition epochs. An area of study most frequently associated 
with SMPs is that of survival analysis and reliability, for which the 
definitive reference is \cite{barlow_prosch65}, and which has been continued by 
the likes of \cite{barbu_etal_04,StocNetModels,SemiMarkov} and others. Of 
special note are the areas of semi-Markov decision processes and 
$PH$-distributions \cite{khar_solo_uluk10,younes_simm04}, often used in 
reliability, but which also appear in the context of SMP first passage moments, 
as in \cite{SMPFirstPass}. Numerical algorithms for the efficient calculation
of first-passage moments have also been studied (e.g., \cite{hunter2016accurate}), and important developments
since \cite{kemeny_snell} are presented in the survey paper by Hunter
\cite{hunter2018}. Other areas that have seen the application of SMP 
models are DNA analysis \cite{barbu}, queueing theory \cite{Kleinrock:1,
neuts_mg1_89}, finance \cite{Janssen:2}, artificial intelligence 
\cite{younes_simm04}, and transportation \cite{brum_etal01,lerman79}, to name 
but a few. 

Explicit time-domain formulas for the first two moments of the first passage 
distribution of a irreducible ergodic SMP with a finite state space have long 
been known. Pyke \cite{pyke1,pyke2} inverted Laplace-Stieltjes transform 
matrices under restrictive non-singularity conditions in order to derive the 
first and second moments. Hunter \cite{hunter1969moments} repeated this 
analysis by means of Markov renewal theory, and then solved for the matrix of 
first passage moments $\bM$ of the SMP through multiplication of the matrix $\bI-\bP$ 
by a generalized inverse, where $\bP$ is the irreducible transition probability matrix of 
the embedded discrete time Markov chain. 
Hunter \cite{hunter2007simple} further demonstrates that $\bm{M}$ may be found using any
generalized inverse $\bm{G}$ for $\bm{I}-\bm{P}$. This criterion directly relates to the main result of this
paper via the action of a specific generalized inverse on the irreducible partition of the
transition matrix of the embedded DTMC.
Although the role of the fundamental 
matrix of the embedded DTMC in solving the problem of finding the first passage 
moments was recognized since at least Kemeny and Snell \cite{kemeny_snell}, it 
was Hunter \cite{hunter1969moments} that recognized its importance by proving 
that the fundamental matrix is a particular generalized inverse for $\bI-\bP$. Some 
years later, Yao \cite{MPFirstPass} was able to use a generalized inverse to 
find \emph{all} moments of first passage for irreducible processes. Zhang and Hou \cite{SMPFirstPass} 
likewise employed a generalized inverse method in order to derive exact first 
passage moments for SMPs with phase- ($PH$-)distributed sojourn times between 
states, thus capitalizing on the robust interest in the reliability community 
for these somewhat exponential-like statistical distributions.

In this article, we derive the formulae for first passage moments in SMPs with finite state space.  These processes can be reducible and/or periodic, the only restriction we impose is that the final state of interest in the first passage is \emph{universally accessible} (UA).  UA states are those that are accessible from each and every other state of the SMP. It turns out that 
the UA-focused method is less involved than earlier methods which, in addition 
to requiring the g-inverse-rendered fundamental matrix, also involves 
the simultaneous computation of \emph{all} first passage times in an 
irreducible process. The proposed approach requires a generous application of linear 
algebra, namely the Perron-Frobenius theorem generalized to reducible matrices 
(and hence reducible processes) in order to arrive at the existence of the 
reduced fundamental matrix. For further details on the Perron-Frobenius 
theorem, and spectral theory in general, see \cite{fiedler2008special}.

The remainder of the paper will proceed as follows.  In Section 2, we define
notation, terminology, and assumptions that guide the remainder of the 
discourse. In section 3, we discuss known results for irreducible processes, 
and hence essential classes in a general SMP. We then present the main result 
in Section 4, which is the derivation of the formula for the first passage 
moments under the condition of universal accessibility. Finally, in Section 5, 
we present a method for estimating the first passage moments of SMPs and a 
brief example.


\section{Notation and Basic Definitions}


In this section we introduce the the notation used in this paper. A boldface 
symbol without indices refers to a matrix (e.g., $\textbf{F}(t)$ is a
matrix with elements $F_{_{ij}}(t)$ in the $i$th row and $j$th column).
We will sometimes drop the function argument for simplicity's sake; e.g.,
$\textbf{F}=\textbf{F}(t)$. In the usual way, we define the Dirac-$\delta$
function as
\begin{equation*}
  \delta_{_{ij}} = \left\{
  \begin{array}{l l}
    0 & \quad \text{if $i \neq j$} \\
    1 & \quad \text{if $i=j$}\\
  \end{array} \right.
\end{equation*}
In addition, we will specify that the $k$-dimensional square matrices
$\bI$ and $\bJ$ denote the identity matrix and the matrix whose
entries consist of ones, respectively. Finally, the matrix binary operator
`$\circ$' denotes Hadamard (element-wise) multiplication; i.e.
\[\big[\textbf{A}\circ\textbf{B}\big]_{_{ij}}=a_{_{ij}}b_{_{ij}}.\]


We introduce semi-Markov processes roughly following the development in \cite{hunter2016accurate}.  Consider a Markov renewal process (MRP) $\{(X_n,T_n),n \geq 0\}$ with finite state space $\stspc=\{1,2,\dots,k\}$ and semi-Markov kernel $\bm{Q}(t)=[Q_{_{ij}}(t)]$, where 
\[Q_{_{ij}}(t)= Pr(X_{n+1}=j, T_{n+1}-T_n \leq t|X_n=i) \]
and $i,j \in \stspc$.  $X_n$ represents the state of the process after the $n^{\text{th}}$ transition and $T_n$ is the time at which the $n^{\text{th}}$ transition occurred.  

Define $\bP$ to be the matrix of transition probabilities composed of
\[p_{_{ij}} \equiv Q_{_{ij}}(+\infty) = Pr(X_{n+1}=j|X_n=i). \]
Also let $\bm{F}(t)$ be the matrix of cumulative distribution functions (CDFs)
\[F_{_{ij}}(t) \equiv Pr(T_{n+1}-T_n \leq t | X_n = i, X_{n+1} = j), \] thus $Q_{_{ij}}(t) = p_{_{ij}}F_{_{ij}}(t)$ or $\bm{Q} \equiv \bP\circ\bF$.  Additionally, define
\[ n_{_{ij}}^{(r)} \equiv \int_{0}^{\infty} t^r dF_{_{ij}}(t)\]
for $r \in \mathbb{N}$.  $n_{_{ij}}^{(r)}$ is the $r^{\text{th}}$ moment of the process sojourn time in state $i$ when transitioning next to state $j$.  Finally, let $\bm{N}^{(r)} \equiv [n_{_{ij}}^{(r)}]$, with $\bm{N} \equiv \bm{N}^{(1)}$ and $n_{_{ij}}^{(1)} \equiv n_{_{ij}}$.  In this article we require $0 < n_{_{ij}} < \infty $ for all $i$ and $j$, when $p_{ij} > 0$. 

Let $\zeta \equiv \{X(t):t\geq 0\}$, where $X(t) \equiv X_n$ for $t \in [T_n,T_{n+1})$.  Thus $\zeta$ is a regular time-homogeneous SMP as a consequence of the above definitions. 



We next address the properties of communication between the states of $\zeta$. 
In \cite{cinlar_intro_stoch}, the notion of connectedness between states is
given in terms of the underlying counting process of the MRP
$\set{\bm{W}(t)}[t>0]$, where $\bm{W}(t)=[W_j(t)]_{j\in S}$ is a vector-valued random
variable for which $W_j(t)$ counts the number of times during the interval
$(0,\,t]$ that $\zeta$ transitions into state $j$. 
If one then defines the
$k\times k$ matrix $\bPsi(t)$ in terms of its $ij^{\text{th}}$ entry
\[\Psi_{ij}(t)=Pr\left( W_j(t)>0 \, \middle| \, X(0)=i \right),\]
then, the accessibility of state $j$ from state $i$, that is, $i\to j$, is 
equivalent to the statement that
\[\Psi_{ij}(\infty)=\lim_{t\to\infty}\Psi_{ij}(t)>0.\]
In other words, the process $\zeta$ will transition from state $i$ to state $j$ 
in finite time with probability 1. Communication between a pair of states $i$ 
and $j$, that is,  $i\to j$ and $j\to i$, will be denoted as $i\tofrom j$. If 
we specify that reflexivity holds for $\tofrom$, that is, every state 
communicates with itself, it can be shown that communication is an equivalence 
relation in the state space $S$. Thus, it follows that $S$ is the disjoint 
union $S=S_1\cup S_2\cup\dots\cup S_K$ of some number $K$ of \emph{essential} 
classes of states whose members communicate exclusively within the class. If 
there exists only one essential class for $\zeta$, then we say that $\zeta$ is 
irreducible. On the other extreme, any state that is the sole member of an 
essential class falls into one of the following categories:
\begin{enumerate}
	\item \emph{Strict Transience}: There is no possibility of returning to the 
	      state once the process leaves it. In other words, if state $j$ is 
	      strictly transient, then
	      \[\Psi_{jj}(+\infty)=0.\]
	\item \emph{Absorption}: A state $j$ is called \emph{absorbing} if it can 
	      only return to itself; that is
	      \[\Psi_{ji}(+\infty)=0,\quad\Psi_{jj}(+\infty)=1,\quad i\neq j.\]
\end{enumerate}
The presence of some $\Psi_{ij}(+\infty)\in(0,\,1)$ indicates that there is some 
state $i'\in S$ for which $i\to i'$, but $i'\nrightarrow j$ ($i'$ may even be 
absorbing). Furthermore, we note that the condition $\Psi_{ii}(+\infty)\in
(0,\,1)$, may be denoted as \emph{non-strict} transience of the state $i$.
Finally, if $\Psi_{ii}(+\infty)=1$, then a return to state $i$ is inevitable, 
and so we must have recurrence of the state $i$. 

This classification of states into recurrent and transient states suggests an 
arrangement of the matrix $\Psi(+\infty)$ in which the component blocks are organized according to the 
normal form $\bP_N$ of the reducible probability matrix $\bP$ of the embedded 
DTMC, which appears as
\begin{equation}\label{eq:canform}
    \bP_N \sim \bPhi\bP\bPhi^{-1} =
    \left(\begin{array}{cccc|cccc}
        \bP_{11}&\bP_{12}&\cdots
            &\bP_{1r}&\bP_{1,r+1}&\bP_{1,r+2}&\cdots &\bP_{1k}\\
        \bz     &\bP_{22}&\cdots
            &\bP_{2r}&\bP_{2,r+1}&\bP_{2,r+2}&\cdots &\bP_{2k}\\
        \vdots&      &\ddots &\vdots&\vdots   &\vdots   &\ddots &\vdots\\
        \bz     &\bz     &\cdots
            &\bP_{rr}&\bP_{r,r+1}&\bP_{r,r+2}&\cdots &\bP_{rk}\\ \hline
        \bz     &\bz     &\cdots &\bz
            &\bP_{r+1,r+1}&\bz      &\cdots &\bz     \\
        \bz     &\bz     &\cdots &\bz     &\bz
            &\bP_{r+2,r+2}&\cdots &\bz     \\
        \vdots&\vdots&\ddots &\vdots&\vdots   &         &\ddots &\vdots\\
            \bz     &\bz     &\cdots &\bz     &\bz
            &\bz          &\cdots &\bP_{kk}\\
    \end{array}\right)
\end{equation}
where $\bPhi$ is a permutation matrix.
The first $r$ blocks along the diagonal correspond to essential classes of 
transient states (both strict and non-strict) while the remaining $k-r$ blocks 
correspond to the recurrent states.

With the relationship between the entries of $\Psi(+\infty)$ and the state 
properties of $\zeta$ having been established, we next introduce yet another 
state-level characteristic that is crucial to determining pairs of states for 
which it is appropriate to compute a first passage moment.
\begin{defn}
    Let $\zeta$ be a SMP with state space $\stspc$. State $j^*\in\stspc$ is 
    said to be \textit{universally accessible} (UA) if, for every $i \in 
    \stspc$, we have $\Psi_{ij^*}(+\infty)=1$.
\end{defn}\noindent
The definition is strong in the sense that, while it is necessary that $i\to 
j^*$ for a UA state $j^*$, it is not \emph{sufficient}. In other words, we must 
achieve certainty with probability 1 that a process starting in any $i\in S$ 
must end up in $j^*$ in a finite period of time. As a matter of terminology, we 
shall refer to a state $j^*$ for which some entries $0\leq\Psi_{ij}(+\infty)<1$
as \emph{sub-UA}. 

It is often convenient to refer to the arrangement of states in a SMP in 
graph-theoretic terms. A \emph{directed graph}, or \emph{digraph}, associated 
to $\bA=[a_{ij}]$, denoted $\grph(\bA)$ is a grouping of \emph{nodes} (states) connected 
by vertices in the set $V(\grph(\bA))=\{1,2,\dots,k\}$ such that the directed 
arc, or edge, $(i,j)$ exists if and only if $a_{ij}>0$. $\grph(\bA)$ is said to 
be \emph{strongly connected} if, for each ordered pair $i,j\in V(\grph(\bA))$, 
there exists a (directed) path in $\grph(\bA)$ from $i$ to $j$. In either case 
of there being an edge or directed path from $i$ to $j$, the implication is 
clearly $i\rightarrow j$. The final connection between irreducibility and 
connectedness is made in the following Proposition:
\begin{prop}
    Let $\bA$ be a nonnegative square matrix. $\bA$ is irreducible if and only 
    if $\grph(\bA)$ is strongly connected.
\end{prop}
\begin{proof}
    See Shao \cite{shao_85}.
\end{proof}

We conclude the section with a Proposition that demonstrates the significance 
of the UA state property, which we assert that, even if it is not greater than 
that of irreducibility, it is evinced at a more fundamental level:
\begin{prop}\label{prop:irr_ua}
    A SMP $\{X(t):t\geq 0\}$ with finite state space $\stspc$ is
    irreducible if and only if every $j\in\stspc$ is UA.
\end{prop}
\noindent The property of a state being UA is, in a sense, the minimal 
requirement for the existence of a vector of finite first-passage moments for that state.


\section{First Passage Times in an Irreducible SMP}

We next address known results for the first passage time moments of an 
irreducible SMP, that is, an SMP for which $\Psi(+\infty)=\bm{E}$ 
We may select an arbitrary state $j\in S$ for which we define the random 
variable
\[Y_{_j}=\inf\{t\geq T_1:X(0)=i,\,X(t)=j\},\,j\in\stspc,\]
where $T_1>0$ is the time of the first transition. Random variable $Y_j$ may 
be described as the time of first passage from an initial state $i$ to state 
$j$ if $i\neq j$, and the time of first return to $j$ otherwise. The 
distribution function $G_{ij}(t)$ of first passage, conditioned on being in the 
initial state $i\in\stspc$, is defined as
\[G_{_{ij}}(t)=Pr\left( Y_{_j} \leq t \, \middle| \, X(0)=i\right),\]
and for which the corresponding $r^{\text{th}}$ moments $m_{_{ij}}^{(r)}$, $r\geq 1$, 
if they exist, are given by
\[m_{_{ij}}^{(r)}\equiv \int_0^{\infty} t^r dG_{ij}(t) = E\left[Y_{_j}^r\, \middle | \,X(0)=i \right].\]
We thus define $\bG(t)$ and $\bm{M}^{(r)}=\left[m_{_{ij}}^{(r)}\right]$ to be 
the matrices of first passage cumulative distribution functions and $r^{\text{th}}$ moments, respectively.

As stated in Proposition 5.15 of \cite[pg104]{RossOpt} and Lemma 4.1 of 
\cite{SMPFirstPass}, the moments of first passage for an irreducible SMP may be 
computed as the finite solution to the systems of equations given by
\begin{align}
    \label{eq:mmt_comps1}
    m_{ij}& \equiv m_{_{ij}}^{(1)} =\sum_{l=1}^k p_{_{il}}
        \left[\left(1-\delta_{_{lj}}\right)m_{_{lj}} + m_{_{il}} \right] \\
    \label{eq:mmt_comps2}
    m_{_{ij}}^{(r)} &= \sum_{l=1}^k p_{_{il}} m_{_{il}}^{(r)} + 
        \sum_{s=1}^{r}\binom{r}{s} \left[ \sum_{l \neq j} p_{_{il}}
            e^{(r-s)}_{_{il}}m_{_{lj}}^{(s)}\right],\qquad r\geq 2.
\end{align}
One or both of these formulae can also be found in \cite{hunter1969moments}, \cite{hunter82}, and \cite{hunter2016accurate} and are well known in the literature. 
Clearly, a necessary condition for $m^{(r)}_{_{ij}}< + \infty$ is that 
$i\rightarrow j$, which is certainly true if the SMP is irreducible. In 
contrast, we observe that $G_{_{ij}}(+ \infty)<1$ (and $m_{_{ij}}= +\infty$) 
might occur for a pair of states $i,\,j\in\stspc$ if $i\nrightarrow j$. As we 
will later show, \eqref{eq:mmt_comps1} and \eqref{eq:mmt_comps2} still hold 
under the somewhat weakened assumption of universal accessibility for the 
terminal state $j$.

The recurrence properties of a SMP may be explained in terms of the 
distribution of the first passage of a SMP from a given state $i\in\stspc$ back 
to itself, otherwise known as the time of (first) return to a state 
$i\in\stspc$. The crucial step is to define
\[H_{_{ii}}=Pr \left( N(Y_{_i}) < +\infty \, \middle| \, Z_0=i \right),\]
which is the probability that the number of steps required for the embedded
DTMC $\{X_n:n\geq 0\}$ to return to state $i$ is finite. If $H_{_{ii}}<1$,
then the state $i\in\stspc$ is called \emph{transient}; otherwise, it is known 
as \emph{recurrent}. If, in addition to recurrence, we have $m_{_{ii}}
< +\infty$, then the state is called \emph{positive recurrent}. The SMP itself is 
deemed, recurrent, transient, or positive recurrent as a \emph{process} if the 
corresponding condition holds for every state $i\in\stspc$. For an irreducible 
SMP with a finite state space, it is well-known that the process is 
automatically positive recurrent. This is not true, in general, for a reducible 
process, but may be evaluated on a state-by-state basis.


\section{First Passage Moments for UA States}

In this section, we derive a formula for determining the first and higher 
moments of first passage times in \emph{reducible} SMPs to single states $j$ 
that are UA. We begin with a technical result that will be needed in the proof 
of Theorem \ref{thm:main} to demonstrate that the matrix formula for the 
moments of first passage to a UA state $j\in\stspc$ is well-defined. For 
notational convenience, define $\bm{e}_{_j}$ to be a vector of length $k$ which contains all zeros except at the $j^{\text{th}}$ position, which is 1.  Also define $\bm{e}$ to be a $k$ length vector of 1s.  
The proof of the following lemma is given in the Appendix.
\begin{lemma} \label{lemma:nonsing_result}
    Let $\{X(t):t\geq 0\}$ be a SMP with finite state space $\stspc$ and
    embedded DTMC at transition epochs with transition probabilities contained
    within the (stochastic) matrix $\bm{P}$. Then the matrix $\left[\bI-\bP+\bP
    \bm{e}_{_j} \bm{e}_{_j}^T\right]$ is nonsingular if and only if state $j\in\stspc$ is
    universally accessible (UA).
\end{lemma}
\noindent Clearly this lemma is true if all states are UA, as shown in Theorem 3.3 of \cite{hunter2007simple}.  Next, we will derive the closed-form analytical expression for the 
$r^{\text{th}}$ first passage moments $\bm{M}^{(r)}=\left[m_{_{ij}}^{(r)}\right]$, for 
$r\geq 1$ and for any given initial state $i\in\stspc$, given that the terminal 
state $j$ is UA.  These results are recursive, thus for $r>1$, one must have the $(r-1)^{\text{th}}$ moments to calculate the $r^{\text{th}}$ moments.  Although a convenient computational feature is that only one inverse matrix must be calculated for any number of moments.
\begin{thm}\label{thm:main}
    Let $\{X(t):t\geq 0\}$ be a regular time-homogeneous SMP with a finite 
    state space $\stspc$. Further suppose that $j\in\stspc$ is UA. Then the 
    $r^{\text{th}}$ moments of the first passage times from all states $i\in\stspc$ to 
    state $j$ contained in the $k$-vector ($k=|\stspc|$)
    \[\bm{M}_{_j}^{(r)}=\left[m_{_{ij}}^{(r)}\right]_{i=1}^k,\quad r\geq 1,\]
    are solutions to the system of equations given by
    \begin{align}
        \label{eq:1st_mmt}
        \bm{M}_{_j}       &\equiv \bm{M}_{_j}^{(1)} = \left[\bI-\bP+\bP\bm{e}_{_j} \bm{e}_{_j}^T \right]^{-1}
            \big(\bP \circ \bm{N} \big) \bm{e} , \\[1ex]
        \notag
        \bm{M}_{_j}^{(r)} &= \left[\bI-\bP+\bP\bm{e}_{_j} \bm{e}_{_j}^T \right]^{-1} \\
        \label{eq:rth_mmt}
        &\times\Bigg[ \left( \bP \circ \bm{N}^{(r)} \right) \bm{e} +
            \sum_{s=1}^{r-1}\binom{r}{s}
            \left[ \left( \bP \circ \bm{N}^{(r-s)} \right)
            \left( \left( \bm{E} - \bm{I} \right)_j \circ \bm{M}_j^{(s)} \right) \right]
            \Bigg], &\text{if }\;r>1,
    \end{align}
    where $\bm{e}$ is a column vector of ones and the scalar entries 
    $m_{_{ij}}^{(1)}$ and $m_{_{ij}}^{(r)}$ for $r\geq 2$ are defined as in 
    \eqref{eq:mmt_comps1} and \eqref{eq:mmt_comps2}, respectively.
\end{thm}

Before proving Theorem \ref{thm:main}, we will compare
this result to Theorem 5.2 of Hunter \cite{hunter82}.
First, we observe that Theorem \ref{thm:main} may be used
to find the mean first passage times to UA states in a SMP with either
an irreducible or reducible transition probability matrix. Hence this
method may be used in place of Hunter's method, albeit for a single
UA state at a time. In the case of a reducible process with a UA state, and
lacking other available alternatives, this method \emph{must}
be used since a unique stationary vector 
will not
exist, in which case Hunter's formula is not well-defined.

Theorem \ref{thm:main} is nevertheless similar to Hunter's formula in
that a generalized inverse of $\bm{I}-\bm{P}$ must be computed.
However, in our proposed method there is no requirement to compute a stationary probability vector.
This advantage may be lost if first-passage moments are required
for many UA states since one must, in each instance, repeat
computations \eqref{eq:1st_mmt} and possibly \eqref{eq:rth_mmt}.  A more in-depth comparison of our Equation \eqref{eq:1st_mmt} to Equation (5.12) of Hunter \cite{hunter82} is included in the appendix (Section \ref{compare}).




\begin{proof} (Theorem \ref{thm:main})
We first show, using induction on the $r^{\text{th}}$ moment, $r\geq 1$, that the system 
of equations \eqref{eq:mmt_comps1} and \eqref{eq:mmt_comps2} give a valid 
relationship between the first-passage moments to a given state $j$ that is UA. 
For the mean time of first passage given by the system \eqref{eq:mmt_comps1}, 
we observe the following at the first transition epoch $T_1$ of the SMP:
\begin{enumerate}
	\item $i\nrightarrow l$ at $T_1\;\Rightarrow$ the corresponding $l^{\text{th}}$ term 
	      drops out of the expression, and
	\item $i\rightarrow l$ at $T_1\;\Rightarrow$ $m_{il}$ and $m_{lj}$ are 
	      well-defined, the latter because $j$ is UA.
\end{enumerate}
We thus conclude that a first-step analysis founded upon the state of the SMP 
at the first transition epoch $T_1$ (c.f. Proposition 5.15 of 
\cite[pg104]{RossOpt}) still holds for a terminal UA state $j$. Next, for the 
induction step, we consider expression \eqref{eq:mmt_comps2} for the $(r+1)^{\text{th}}$ 
moment, where $r\geq 1$. We likewise claim that the original renewal argument 
given in Lemma 4.1 of \cite{SMPFirstPass} for the derivation of 
\eqref{eq:mmt_comps2} for the $r^{\text{th}}$ moments of first passage is valid. In order 
to see this, we rewrite, for $i\in\stspc$, expression \eqref{eq:mmt_comps2} as
\begin{equation}\label{eq:r+1_mmt}
    m_{_{ij}}^{(r+1)}=\sum_{l=1}^k p_{_{il}}
        \left[\left(1-\delta_{_{lj}}\right)m_{_{lj}}^{(r+1)} + 
        m_{_{ik}}^{(r+1)}\right] + \sum_{s=1}^{r}\binom{r}{s} \left[ \sum_{l \neq j} p_{_{il}}
            n_{_{il}}^{(r-s)}m_{_{lj}}^{(s)}\right]. 
\end{equation}
The inductive hypothesis and items 1) and 2) above guarantee that the last sum in \eqref{eq:r+1_mmt} 
is 
well-defined while the remaining part is in exactly the same 
form as \eqref{eq:mmt_comps1}, which has just been shown to have a finite 
solution via the base step.

Thus, for arbitrary $i\neq j$, where $i\in\stspc$, we may transform 
\eqref{eq:mmt_comps1} into the equivalent matrix expression
\[ \bm{M} = [m_{_{ij}}] = \bP \big( \left( \bJ - \bI \right)  \circ
   \bm{M} \big) +\big( \bP \circ \bm{N} \big) \bJ \,. \]
In this form we are not able to solve directly for $\bm{M}$, but, under the 
assumption that $j$ is a specific UA state in $\stspc$, it is possible to solve 
for the $j^{th}$ column of $\bm{M}$, which we denote as $\bm{M}_{j}$. We then obtain,
\[ \bm{M}_{j} = \bP \left[ \big( \left( \bJ - \bI \right)
   \circ \bm{M} \big) \right]_{j} + \big( \bP \circ \bm{N} \big)
   \bm{e} \, .\]
Next, we isolate $(\bm{P} \circ \bm{N}) \, \bm{e}$ so that
\[ \bm{M}_{j} - \bP \left[ \big( \left( \bJ - \bI \right)
   \circ \bm{M} \big) \right]_{j} = \big( \bP \circ \bm{N} \big)
   \bm{e} \, .\]
Factoring out $\bm{M}_j$ gives
\[ \left[\bI-\bP+\bP\bm{e}_{_j} \bm{e}_{_j}^T \right]\bm{M}_{j} = \big( \bP
   \circ \bm{N} \big) \bm{e} \, \]
which allows us to finally solve for $\bm{M}_j$ as
\[\bm{M}_{j} = \left[\bI-\bP+\bP\bm{e}_{_j} \bm{e}_{_j}^T \right]^{-1} \big(
 \bP \circ \bm{N} \big) \bm{e}  \,.\]
The previous step is justified since $\bI-\bP+\bP\bm{e}_{_j} \bm{e}_{_j}^T$ is nonsingular (see Lemma \ref{lemma:nonsing_result}).
This proves that \eqref{eq:1st_mmt} is, indeed, well-defined.

A general formula for the $r^{\text{th}}$ moment, where $r\geq 2$, is given in Lemma
4.1 of \cite{SMPFirstPass} as
\[m_{_{ij}}^{(r)} = \sum_{l=1}^{k} p_{_{il}} n_{_{il}}^{(r)} + \sum_{s=1}^{r}
\binom{r}{s} \left[ \sum_{l \neq j} p_{_{il}} n^{(r-s)}_{il} m_{_{lj}}^{(s)}\right],\]
which is expressed in matrix notation as
\begin{equation*}
\bm{M}^{(r)} = \left( \bP \circ \bm{N}^{(r)} \right) \bJ
    + \sum_{s=1}^{r} \binom{r}{s}\left[ \left( \bP \circ \bm{N}^{(r-s)} \right)
    \left( \left( \bJ - \bI \right)\circ \bm{M}^{(s)} \right)  \right] \, .
\end{equation*}
Solving for the $j^{\text{th}}$ column gives
\begin{equation*}
\bm{M}^{(r)}_j = \left( \bP \circ \bm{N}^{(r)} \right) \bm{e}
    + \sum_{s=1}^{r} \binom{r}{s} \left[ \left( \bP \circ \bm{N}^{(r-s)} \right)
    \left( \left( \bJ - \bI \right)\circ \bm{M}^{(s)} \right)_j  \right] .
\end{equation*}
Using $\bm{N}^{(0)}=\bJ$ (the identity under the Hadamard product), we extract the 
$r$th term of the summation to obtain
\begin{align*}
\bm{M}^{(r)}_j - \bP  \left[ \left( \bJ - \bI \right)\circ \bm{M}^{(r)} \right]_j
    = \left( \bP \circ \bm{N}^{(r)} \right) \bm{e} +   \sum_{s=1}^{r-1} \binom{r}{s}
    \left[ \left( \bP \circ \bm{N}^{(r-s)} \right) \left( \left( \bJ - \bI
    \right)\circ \bm{M}^{(s)} \right)_j  \right].
\end{align*}
We further observe that
\[\bm{M}^{(r)}_j - \bP  \left[ \left( \bJ - \bI \right)\circ
  \bm{M}^{(r)} \right]_j = \left(\bI-\bP+\bP\bm{e}_{_j} \bm{e}_{_j}^T \right)\bm{M}_{j}^{(r)},\]
which gives
\begin{equation*}
\left(\bI-\bP+\bP\bm{e}_{_j} \bm{e}_{_j}^T \right)\bm{M}_{j}^{(r)} = \left( \bP \circ \bm{N}^{(r)}\right)
    \bm{e} + \sum_{s=1}^{r-1} \binom{r}{s} \left[ \left( \bP \circ \bm{N}^{(r-s)}
    \right) \left( \left( \bJ - \bI \right)\circ \bm{M}^{(s)} \right)_j  \right].
\end{equation*}
Finally, we solve for $\bm{M}_j^{(r)}$ to obtain
\[\bm{M}_{j}^{(r)} = \left[\bI-\bP+\bP\bm{e}_{_j} \bm{e}_{_j}^T \right]^{-1}
    \Bigg[ \left( \bP \circ \bm{N}^{(r)} \right) \bm{e} + \sum_{s=1}^{r-1}
    \binom{r}{s} \left[ \left( \bP \circ \bm{N}^{(r-s)} \right)
    \left( \left( \bJ - \bI \right)_j \circ \bm{M}_j^{(s)} \right) \right]
    \Bigg] \, .\]
As argued in the proof of formula, \eqref{eq:1st_mmt}, the inverse 
$\left[\bI-\bP+\bP\bm{e}_{_j} \bm{e}_{_j}^T\right]^{-1}$ exists. Hence, \eqref{eq:rth_mmt} is 
also well-defined.
\end{proof}

\noindent We next investigate some statistical aspects, using Theorem
\ref{thm:main}, to estimate the first passage moments to universally accessible
states in a SMP.


\section{Estimation}

In this section we will derive consistent estimates for first passage
moments in SMPs. Since the SMP $\{X(t):t\geq 0\}$ is time-homogeneous, we assume 
without loss of generality that $X(0)=i\in\stspc$. If we observe the SMP for a 
period of time $T>0$, then, for any $j\in\stspc$, we may then define the point 
estimators $\hat{p}_{ij}$ and $\hat{n}^{(r)}_{ij}$ for the probability and the
$r^{\text{th}}$ moment of the sojourn time of the SMP as it transitions from $i$ to $j$,
respectively. They are defined as follows:
\[\hat{p}_{_{ij}}\equiv\frac{u_{_{ij}}}{\sum_{l\in\stspc}{u_{_{il}}}},\qquad
  \hat{n}^{(r)}_{_{ij}}\equiv\frac{1}{u_{_{ij}}}\sum_{l=1}^{u_{_{ij}}} x_{_{ijl}}^r,
  \qquad i,j\in\stspc,\]
where
\begin{align*}
    u_{_{ij}}  &\equiv \text{ Number of observed transitions from state $i$ to 
        state $j$ by time $T$}, \\
    x_{_{ijl}} &\equiv \text{ $l^{\text{th}}$ observed sojourn time from state $i$ to 
        state $j$ by time $T$}.
\end{align*}
We further assume $T$ is large enough so that at least one transition from 
$i$ to $j$ has been observed; in other words, $u_{_{ij}}\geq 1$. 
In order to make inferential hypotheses using these estimators, it is necessary 
to first show that they are consistent. A point estimator $\hat{\theta}_n$ is 
said to be \emph{consistent} if it converges in probability to the true 
population statistic $\theta$ as the sample size $n$ increases; that is, for 
each $\epsilon>0$,
\[\lim_{n\to\infty}P\left(|\hat{\theta}_n-\theta|<\epsilon\right)=1.\]
This condition is written in shorthand as
\[\hat{\theta}_n\stackrel{P}{\rightarrow}\theta.\]
We now show that this condition holds for the matrix estimators 
$\hat{\bP}\equiv[\hat{p}_{ij}]$ and $\hat{\bm{N}}\equiv[\hat{n}_{ij}]$.
\begin{lemma}
    The matrix estimators $\hat{\bP}$ and $\hat{\bm{N}}$ are consistent.
\end{lemma}
\begin{proof}
    Let $\{V_l\}$ be a sequence of Bernoulli random variables such that $V_l=1$ 
    when a transition from $i$ to $j$ occurs at the $l^{\text{th}}$ transition, and is 0 
    otherwise. Accordingly, if $U>0$ transitions are observed in the time 
    interval $(0,T]$, then the estimated probability of transition from $i$ to 
    $j$ becomes
    \[\hat{p}_{ij}=\frac{1}{U}\sum_{l=1}^U{V_l},\]
    with the following equivalences
    \[U=\sum_{l\in\stspc}{u_{_{il}}},\qquad u_{_{ij}}=\sum_{l=1}^U{V_l}.\]
    The Markov property at transitions of the embedded DTMC of the SMP implies 
    that the $V_l$ are independent and identically distributed (i.i.d.) random 
    variables. Hence, by the Weak Law of Large Numbers (see Theorem 5.5.2 of 
    \cite[p.232]{CasBerg}), we have
    \[\hat{p}_{ij}\stackrel{P}{\rightarrow}p_{_{ij}},\]
    which demonstrates consistency.

    Likewise, we see that the $x_{ijl_1}$ are independent of $x_{ijl_2}$ so 
    long as $l_1\neq l_2$. Thus, the collection $\{x_{ijl}\}_{l=1}^{u_{ij}}$ 
    is i.i.d. By the same reasoning as above, we obtain the convergence in 
    probability
    \[\hat{n}_{ij}\stackrel{P}{\rightarrow}n_{_{ij}},\]
    Hence, the $\hat{n}_{ij}$ are consistent.
\end{proof}
We are now in a position to define the estimators of the $r^{\text{th}}$ moments of first 
passage from state $i$ to state $j\in\stspc$. By replacing $\bP$ and $\bN$ with 
the matrix estimators $\hat{\bP}$ and $\hat{\bN}$, respectively, in formulas 
\eqref{eq:1st_mmt} and \eqref{eq:rth_mmt}, we obtain the natural estimators
\begin{align}
    \label{eq:mu_est1}
    \widehat{\bm{M}}_{j} &\equiv \left[\bI-\widehat{\bP}+\widehat{\bP}\bm{e}_{_j} \bm{e}_{_j}^T
        \right]^{-1} \big( \widehat{\bP} \circ \widehat{\bN} \big) \bJ_{j}, \\[1.5ex]
    \notag
    \widehat{\bm{M}}_{j}^{(r)} &\equiv \left[\bI-\widehat{\bP}+\widehat{\bP}\bm{e}_{_j} \bm{e}_{_j}^T \right]^{-1} \\
    \label{eq:mu_est>1}
                   &\times\Bigg[ \left(\widehat{\bP} \circ \widehat{\bN}^{(r)} \right) 
        \bJ_j + \sum_{s=1}^{r-1}\binom{r}{s} \left[ \left( \widehat{\bP} \circ 
        \widehat{\bN}^{(r-s)} \right) \left(
        \left( \bJ - \bI \right)_j\circ \widehat{\bm{M}}_j^{(s)} \right)  \right]
        \Bigg], \quad r\geq 2.
\end{align}
As expected, estimators \eqref{eq:mu_est1} and \eqref{eq:mu_est>1} are also 
consistent.
\begin{lemma}\label{lemma:mu_consist}
    For a state $j\in\stspc$ that is UA with respect to the digraph 
    $\grph(\widehat{\bP})$, the estimators $\widehat{\bm{M}}_{j}^{(r)}$, $r\geq 1$, 
    are consistent.
\end{lemma}
The assumption that $j$ is UA with respect to $\grph(\widehat{\bP})$ as 
stated in Lemma \ref{lemma:mu_consist} addresses a possible issue when 
estimating first passage moments in that the process defined by $\widehat{\bP}$. 
This concerns the possibility that, due to numerical issues or lack of 
sufficient observations, the estimated probabilities $\widehat{\bP}$ might
suggest that state $j$ is \emph{not} UA. If this is the case, then Lemma 
\ref{lemma:mu_consist} may not be used to estimate the first passage moments 
to state $j$. One possible workaround is to replace the anomalous zero 
probabilities with small positive values, then proceed with procedures outlined 
earlier. This may, of course, introduce inestimable inaccuracies into the 
computation. Another approach would be to simply delete states that become 
disconnected from $j$, but, again, the same concerns with regard to obtaining 
an accurate estimate would potentially arise.


\section{Example}
We give an example of an SMP and show how the first passage moments can be
estimated.  Therefore, given the process depicted in Figure
\ref{fig:RevIllDeathModel3} we have 3 transition distributions and a
probability $p$.  We will calculate the first passage moments using the direct
transition moments, $\bm{N}$.

\begin{figure}[ht]
\begin{center}
\includegraphics[width=4in]{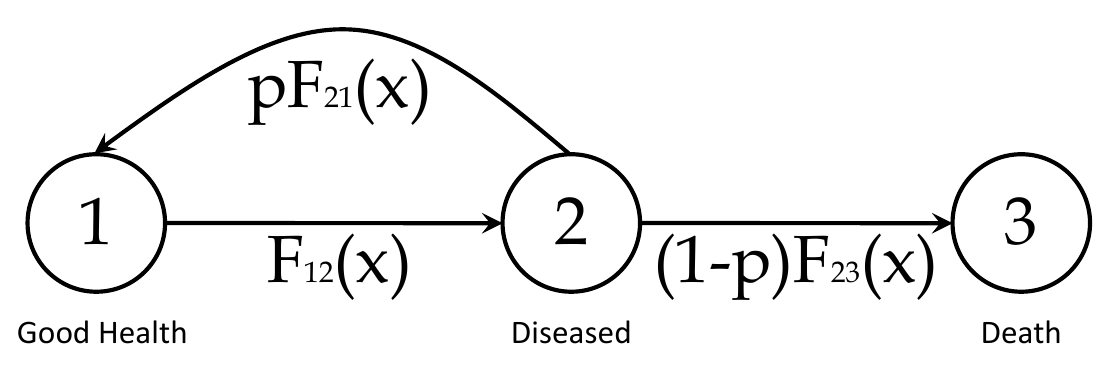}
\caption{A SMP representation of the progress of a medical patient.}
\label{fig:RevIllDeathModel3}
\end{center}
\end{figure}

We begin the procedure with
\begin{equation*}
\bP = \left[ \begin{array}{ccc}  0&1&0\\ p&0&1-p\\ 0&0&1\\ \end{array}
    \right]\text {  and  }
\bN = \left[ \begin{array}{ccc}  0&n_{12}&0\\ n_{21}&0&n_{23}\\ 0&0&n_{33}
    \\ \end{array} \right].
\end{equation*}
The expected first passage vector $\bm{M}_3$ to State 3 is then obtained as
\begin{eqnarray*}
\bm{M}_3 &=&
\left( \left[ \begin{array}{ccc}  1&0&0\\ 0&1&0\\ 0&0&1\\ \end{array}
    \right] -
\left[ \begin{array}{ccc}  0&1&0\\ p&0&0\\ 0&0&0\\ \end{array}
    \right]\right)^{-1}
\left(  \left[ \begin{array}{ccc}  0&n_{12}&0\\ p \, n_{21}&0& (1-p)
    \,n_{23}\\ 0&0&n_{33}\\ \end{array} \right] \right)
\left[ \begin{array}{c}  1\\ 1\\ 1\\ \end{array} \right]  \\
&=& \frac{1}{1-p}\left[ \begin{array}{ccc}  1&1&0\\ p&1&0\\ 0&0&1-p\\
                        \end{array} \right]
                 \left[ \begin{array}{c}  n_{12}\\ p \, n_{21} + (1-p) \,
                            n_{23}\\ n_{33} \\
                        \end{array} \right] \\
&=& \frac{1}{1-p}\left[ \begin{array}{c}  n_{12} + p \, n_{21} + (1-p) \,
                            n_{23}\\                                p \,
                            n_{12} + p \, n_{21} + (1-p) \, n_{23} \\
                            (1-p) \, n_{33} \\
                        \end{array} \right] \\
\end{eqnarray*}
Looking closely at these values we can see they make sense.  As $p$ gets
small we can see $m_{13} \rightarrow  n_{12} + n_{23}$ and $m_{23}
\rightarrow n_{23}$.  This simple example demonstrates the theory discussed
earlier and how for even large systems find the first passage moments is only
constrained by the computational burden of computing the inverse of
$\bI-\bP+\bP\bm{e}_{_j} \bm{e}_{_j}^T$.

If numbers are provided then numerical computer programs can handle these
types of problems with ease.  Now suppose we have
\begin{equation*}
\bP = \left[ \begin{array}{ccc}  0&1&0\\ 0.8&0&0.2\\ 0&0&1\\
             \end{array}\right] \text {  and  }
\bN = \left[ \begin{array}{ccc}  0&6&0\\ 0.7&0&1.1\\ 0&0&0\\
             \end{array}
\right].
\end{equation*}
We get the following result
\begin{equation*}
\bm{M}_3 = \left[ \begin{array}{c}  33.9\\ 27.9\\ 0\\ \end{array} \right].
\end{equation*}
The R-code for this example is included in the appendix.  We feel the
methods presented in this paper provide a comprehensive method to determine
the first passage moments of a SMP.

\section{Conclusion}
In this paper we have devised an exact time-domain approach to derive the
moments $m_{_{ij}}$ of first passage time distributions in \emph{reducible}
SMPs, given that the terminal state $j$ fulfills the conditions of universal
accessibility. We note that the solution method presented here may be used to 
find the first passage moments of an \emph{irreducible} SMP as well. This has 
the advantage of requiring the simultaneous solution of first passage moments to 
only \emph{single} UA states $j$, rather than to all states, thereby reducing 
the computational load, particularly for large SMPs. We have also demonstrated 
the existence of consistent point estimators for the first passage moments of 
processes that may be modeled as SMPs.

\section{Acknowledgements}
The authors thank the editors and reviewers for providing feedback which improved the work.


\section{Appendix}

\subsection{Properties of a Stochastic Matrix}

The Perron-Frobenius theorem adapted to finite-dimensional irreducible and 
nonnegative matrices is very useful for characterizing the set of eigenvalues 
of such matrices. As we will see later, the theory may be (indirectly) extended 
to even reducible nonnegative matrices by leveraging their distinctive normal 
form. Let $\bA\in\Rp^{k\times k}$ for some positive integer $k$. We define the 
\emph{spectrum} of $\bA$, denoted $\sigma_{\bA}$, to be the set of its 
eigenvalues. Its \emph{spectral radius}, denoted $\rho(\bA)$, is given by
\[\rho(\bA)=\max\{|\lambda|:\lambda \in \sigma_{\bA} \}\in\Rp,\]
which indicates the maximum radius of the disc that contains $\sigma_{\bA}$ in 
the complex plane. Of particular interest is the case of a finite-dimensional 
\emph{stochastic matrix} $\bA$, which is a nonnegative square matrix such that 
$\bA \, \bm{e}=\bm{e}$, where $\bm{e}$ is a column vector of ones. Perron-Frobenius 
theory, via Proposition \ref{prop:3} for the reducible case, implies that the 
spectral radius is likewise an eigenvalue of $\bA$, denoted the \emph{Perron 
root of} $\bA$. Stochastic matrices comprise the boundary of the unit ball 
$\mathcal{A}=\{\bA\in\Rp^{k\times k}:||\bA||_{\infty}\leq 1\}$ of 
finite-dimensional nonnegative matrices in the normed linear space induced by 
the \emph{infinity norm} $||\cdot||_{\infty}$, which is given by the maximum 
absolute row sum of $\bA=[a_{_{ij}}]$, or
\[||\bA||_{\infty}=\max_i\sum_{j=1}^k{|a_{_{ij}}|}=\max(A\,\bm{e}).\]
As the next Proposition will show, we may classify certain elements of
$\bA\in\mathcal{A}$ with spectral radius $\rho(\bA)<1$ as \emph{substochastic}, 
which is to say that $0<\min(\bA\,\bm{e})<1$.

\begin{prop} \label{prop:1}
Suppose that $\bA\in\mathcal{A}$. If $\rho(\bA) < 1$, then $\bA$ is 
substochastic.
\end{prop}
\begin{proof}
    Clearly, since $\bA\in\mathcal{A}$, it must be either stochastic or
    substochastic. Therefore the only thing that must be proved is that $\bA$ 
    is not stochastic. Assume $\bA$ is stochastic; i.e. $\bA\,\bm{e} = \bm{e}$. 
    This implies $1$ is an eigenvalue, which contradicts $\rho(\bA) < 1$. 
    Therefore, $\bA$ must be substochastic.
\end{proof}

\noindent For an irreducible nonnegative matrix $\bA$, it is, in fact,
sufficient for $\bA$ to have a spectral radius that is strictly less than
unity in order to be substochastic, as the next Proposition shows.

\begin{prop} \label{prop:2}
    If $\bA\in\Rp^{k\times k}$ is an irreducible substochastic matrix, then 
    $\rho(\bA) < 1$.
\end{prop}
\begin{proof}
    See Theorem 7 in \cite{kolpakov1983matrix}.
\end{proof}

The following Proposition relates the spectral radius of the sub-blocks of the
matrix in normal form to that of the entire matrix.

\begin{prop} \label{prop:3}
    Suppose $\bA\in\Rp^{k\times k}$ is a reducible matrix in normal form. 
    Then $\rho(\bA)=\max_{\nu} \rho(\bA_{\nus})$ for $1 \leq \nu \leq K$.
\end{prop}
\begin{proof}
    See Lemma 1 in \cite[pg. 303]{kincaid2002numerical} with an additional
    induction argument to get the result or as argued in 
    \cite[pg.115]{fiedler2008special}.
\end{proof}


\subsection{Proof of Lemma \ref{lemma:nonsing_result}}

\begin{proof}
We begin with the observation that, since $\bA=[\bA_{\nu\gk}]=\bP-\bP \bm{e}_{_j} \bm{e}_{_j}^T $ is
formed by setting each element of the $j^{\text{th}}$ column of $\bP$ to 0, we
essentially remove all directed arcs $(i,j)$ in the digraph $\grph(\bP)$
for each $i\in V(\grph(\bP))$ in order to produce $\grph(\bA)$.
This means that $\grph(\bA)$ cannot be strongly connected, and thus
$\bA$ must be reducible. We may therefore assume that $\bA$ is in canonical
form \eqref{eq:canform}. Furthermore, because the $j^{\text{th}}$ column is zero, we
will assume without loss of generality that the canonical form of $\bA$
corresponds to the particular ordering of the states in $\stspc$ in which
state $j$ is re-designated as state $1$. We impose the same permutation and
partitioning on $\bP=[\bP_{\nu\gk}]$ so that
\begin{equation}\label{eq:A_p_comp}
    \bA_{\nu\gk}=\begin{cases}
               \bP_{\nu\gk} & \text{if }(\nu,\gk)\in\{1,\dots,K\}\times
                   \{2,3,\ldots,K\}, \\
               \bz        & \text{if }(\nu,\gk)\in\{1,\dots,K\}\times\{1\},
           \end{cases}
\end{equation}
where, as in \eqref{eq:canform}, $K$ is the dimension of $\bA$. 
Notice that since $\bP$ may be irreducible, the above does not necessarily
imply that $\bP$ can be put in canonical form, but rather is element-wise
equivalent to $\bA$, save for the first column, which, unlike that of $\bA$, may 
contain positive entries. Stated succinctly, we have that
\[\bz=\bA_{\nu 1}\leq\bP_{\nu 1},\qquad\nu=1,\dots,K.\]

Assume that $\left[\bI-\bA \right]$ is nonsingular, which directly implies
that $1\notin\sigma_{\bA}$; that is, $1$ is not an eigenvalue of $\bA$. Since
$\bP$ is a row-stochastic matrix, and because of the equivalence given in
\eqref{eq:A_p_comp}, the Gerschgorin Circle Theorem (see \cite[Eqn. 
7.1.13]{meyer2000matrix}) indicates that the spectral radius 
$\delta=\rho(\bA)\leq 1$. Furthermore, the nonnegativity of $\bA$ permits the 
use of Equation 8.3.1 of \cite{meyer2000matrix} to then assert that the Perron 
root $0\leq\delta\leq 1$ exists. However, since we have shown that $1\notin
\sigma_{\bA}$, it must then be the case that $\delta<1$. This implies by 
Proposition \ref{prop:3} that $\rho(\bA_{\nus})<1$ for all  $\nu \in 
\{1,\dots,K\}$ and hence, by Proposition \ref{prop:2}, each diagonal block 
$\bA_{\nus},\;\nu\in\{1,\dots,K\}$ must be substochastic.

We now consider the $\nu$th diagonal block in the canonical form of $\bA$,
where $\nu\in\{2,\dots,K\}$, and proceed to show that each state $i$ associated 
to the vertex set $V(\grph(\bA_{\nus}))$ can access state $1$. Because $\bP$ 
is a row-stochastic matrix and $A_{\nus}$ is substochastic, either or both of 
the following may hold:
\begin{enumerate}
    \item $\bP_{\nu 1} \neq \bz$,  or
    \item $\bA_{\nu\gk} \neq \bz$ for some $\gk > \nu$.
\end{enumerate}
For 1), $\bP_{\nu 1} \neq \bz$ indicates the existence of states $i_{\nu}\in 
V(\grph(\bA_{\nus}))$ (with $i_{\nu}=i$ possible, but not necessary) and $1\in 
V(\grph(\bA_{1 1}))$ for which there is a directed arc $(i_{\nu},1)$. Moreover, 
the irreducibility of $A_{\nus}$ gives a directed path from $i$ to $i_{\nu}$. 
We thus obtain
\[i\rightarrow i_{\nu}\rightarrow 1. \]
In other words, there is a directed path from $i$ to 1.

If 2) holds, there exists a directed arc from some state $i_{\nu}\in 
V(\grph(\bA_{\nus})$ (again, with the possibility that $i_{\nu}=i$) to a state 
$i_{\gk}\in V(\grph(\bA_{\gks}))$. From here, we are again confronted with 
choices 1) and 2). If 1) holds, then the previous argument gives us a directed 
path from $i_{\gk}$ to 1. Since the irreducibility of $\bA_{\nus}$ implies the 
existence of a path from $i$ to $i_{\nu}$, we have the accessibility chain
\[i\rightarrow i_{\nu}\rightarrow i_{\gk}\rightarrow 1,\]
and we are done. Otherwise, we proceed to the next diagonal block following 
$\bA_{\gks}$ and continue until we reach a state $i_K\in V(\grph(\bA_{KK}))$ in 
the last diagonal block $\bA_{KK}$. The only choice here, due to the this block 
being substochastic, is 1); that is, $\bP_{K1}\neq\bz$, for which we have 
already demonstrated the existence of the connection $i_K\rightarrow 1$.
Each of the preceding paths may then be combined to form a single directed
path from an arbitrarily selected $i\in V(\grph(\bA_{\nus}))$ to $1$ so that
\[i\rightarrow i_{\nu}\rightarrow i_{\gk}\rightarrow\dots\rightarrow
  i_K\rightarrow 1.\]
Thus, state $1$ is UA.

For the reverse implication, we will assume that state $1$ is UA, and then
proceed to show that $\left[\bI-\bA\right]$ is nonsingular. The reducibility of
$\bA$ allows us to assume that it possesses canonical form and, furthermore,
that each submatrix on the diagonal of the canonical matrix corresponding to
$\bA$ is irreducible or zero. Consider an arbitrary nonzero, and hence
irreducible, diagonal submatrix $\bA_{\nus}$ for some $\nu\in\{2,\dots,K\}$
(recall that $\bA_{11}=\bz$ by definition of $\bA$). By the assumption that
state $1$ is UA, there must be a directed path from each state in the vertex
set $V(\grph(\bA_{\nus}))$ to 1, which in turn implies that $\bA_{\nus}$
is substochastic. By Proposition \ref{prop:2}, $\rho(\bA_{\nus})<1$. Using
this fact, and the fact that the spectral radii of the zero submatrix blocks
are 0, we may invoke Proposition \ref{prop:3}, to state that $\rho(\bA)<1$.
Hence, $\left[\bI-\bA\right]$ is nonsingular, which completes the proof.
\end{proof}

\subsection{Comparing Equation \eqref{eq:1st_mmt} with Previous Results} \label{compare}

To compare our main result, Equation \eqref{eq:1st_mmt}, with Equation (5.12) of Hunter \cite{hunter82} we start by stating Equation (5.12) in our notation.
\begin{align} \label{eq:hunter5.12}
\bm{M} = \Big[ & \bm{G}\left(\bm{P} \circ \bm{N} \right) \bm{e} \bm{\pi}^{T} - \bm{E} \left(\bm{I} \circ \left( \bm{G}\left(\bm{P} \circ \bm{N} \right) \bm{e} \bm{\pi}^{T} \right)  \right) + \nonumber \\ & \bm{\pi}^{T} \left( \bm{P} \circ \bm{N}  \right) \bm{e} \left( \bm{I}-\bm{G}+ \bm{E} \left( \bm{I} \circ \bm{G} \right)  \right)   \Big] \left[ \bm{I} \circ \left( \bm{e} \bm{\pi}^T \right) \right]^{-1}.
\end{align}
\noindent In Equation \eqref{eq:hunter5.12}, $\bm{G}$ represents any generalized inverse of $\bm{I}-\bm{P}$ and $\bm{\pi}$ is the stationary distribution of the process.  However, in our methodology we require a specific generalized inverse of $\bm{I}-\bm{P}$, which is 
\begin{equation} \label{geninveq}
\left[\bI-\bP+\bP\bm{e}_{_j} \bm{e}_{_j}^T \right]^{-1}.
\end{equation}
In our methodology we only find the first passage moments from all states to state $j$ (assuming $j$ is UA).  Using Equation \eqref{eq:hunter5.12} one must assume an irreducible process. 

It can be shown if $\bm{G}$ is a generalized inverse of $\bm{I}-\bm{P}$ in the form of Equation \eqref{geninveq}, the $j^{\text{th}}$ column of the matrix
\begin{equation*}
- \bm{E} \left(\bm{I} \circ \left( \bm{G}\left(\bm{P} \circ \bm{N} \right) \bm{e} \bm{\pi}^{T} \right)  \right) + \bm{\pi}^{T} \left( \bm{P} \circ \bm{N}  \right) \bm{e} \left( \bm{I}-\bm{G}+ \bm{E} \left( \bm{I} \circ \bm{G} \right)  \right)
\end{equation*}
is identically zero.  Removing those terms, substituting Equation \eqref{geninveq} for $\bm{G}$, and eliminating all but the $j^{\text{th}}$ column of Equation \eqref{eq:hunter5.12} yields 
\begin{align*} 
\bm{M}_j &= \left(\left[ \left[\bI-\bP+\bP\bm{e}_{_j} \bm{e}_{_j}^T \right]^{-1}\left(\bm{P} \circ \bm{N} \right) \bm{e} \bm{\pi}^{T}   \right] \left[ \bm{I} \circ \left( \bm{e} \bm{\pi}^T \right) \right]^{-1}\right)_j \\
 & = \left[\bI-\bP+\bP\bm{e}_{_j} \bm{e}_{_j}^T \right]^{-1}\left(\bm{P} \circ \bm{N} \right) \bm{e}.
\end{align*}
This last equation is identical to our Equation \eqref{eq:1st_mmt}.  This demonstrates how our methodology ties into the previous literature, which assumes an irreducible process.

\subsection{R-Code for Example}

\begin{verbatim}
P <- matrix(c(0,1,0,.8,0,.2,0,0,1),ncol=3,nrow=3,byrow=T)
N <- matrix(c(0,6,0,.7,0,1.1,0,0,0),ncol=3,nrow=3,byrow=T)
I <- diag(3)
e3 <- c(0,0,1)
e <- matrix(1,ncol=1,nrow=3)
m3 <- solve(I-P+P%*%e3%*%t(e3))%*%(P*N)%*%e
m3
\end{verbatim}

\newpage


\bibliographystyle{amsplain}
\bibliography{mybib}

\providecommand{\bysame}{\leavevmode\hbox to3em{\hrulefill}\thinspace}
\providecommand{\MR}{\relax\ifhmode\unskip\space\fi MR }
\providecommand{\MRhref}[2]{%
  \href{http://www.ams.org/mathscinet-getitem?mr=#1}{#2}
}
\providecommand{\href}[2]{#2}
\begin{thebibliography}{10}

\bibitem{barbu_etal_04}
V.~Barbu, M.~Boussemart, and N.~Limnios, \emph{Discrete-time semi-{M}arkov
  model for reliability and survival analysis}, Commun Stat-Theor M \textbf{33}
  (2004), no.~11, 2833=962868.

\bibitem{barbu}
Vlad~Stefan Barbu and Nikolaos Liminios, \emph{Semi-{M}arkov chains and hidden
  semi-{M}arkov models toward = applications: Their use in reliability and
  {DNA} analysis}, Springer, New York, 2008.

\bibitem{barlow_prosch65}
R.E. Barlow and F.~Proschan, \emph{Mathematical theory of reliability}, John
  Wiley \& Sons, London and New York, 1965.

\bibitem{brum_etal01}
Shelby Brumelle and Darius Walczak, \emph{Dynamic airline revenue management
  with multiple semi-{M}arkov demand}, Operations Research \textbf{51} (2003),
  no.~1, 137--148.

\bibitem{StocNetModels}
Ronald~W. Butler and Aparna~V. Huzurbazar, \emph{Stochastic network models for
  survival analysis}, J. Amer. Statist. Assoc. \textbf{92} (1997), 246--257.
  \MR{MR1436113 (98e:62140)}

\bibitem{CasBerg}
George Casella and Roger~L. Berger, \emph{Statistical inference}, second ed.,
  Duxbury, Pacific Grove, CA., 2002.

\bibitem{cinlar_intro_stoch}
E.~{\c{C}}{\i}nlar, \emph{Introduction to stochastic processes}, Prentice-Hall,
  Englewood Cliffs, NJ, 1975.

\bibitem{DAmico2006}
Guglielmo D'Amico, Jacques Janssen, and Raimondo Manca, \emph{{Homogeneous
  semi-Markov reliability models for credit risk management}}, Decisions in
  Economics and Finance \textbf{28} (2006), 79--93.

\bibitem{fiedler2008special}
Miroslav Fiedler, \emph{Special matrices and their applications in numerical
  mathematics}, 2nd ed., Dover Publications, New York, 2008.

\bibitem{hunter1969moments}
Jeffrey~J. Hunter, \emph{On the moments of {M}arkov renewal processes},
  Advances in Applied Probability \textbf{1} (1969), no.~2, 188--210.

\bibitem{hunter82}
\bysame, \emph{Generalized inverses and their application to applied
  probability problems}, Linear Algebra and its Applications \textbf{45}
  (1982), 157--198.

\bibitem{hunter2007simple}
\bysame, \emph{Simple procedures for finding mean first passage times in
  {M}arkov chains}, Asia-Pacific Journal of Operational Research \textbf{24}
  (2007), no.~6, 813--829.

\bibitem{hunter2016accurate}
\bysame, \emph{Accurate calculations of stationary distributions and mean first
  passage times in {M}arkov renewal processes and markov chains}, Special
  Matrices \textbf{4} (2016), no.~1, 151--175.

\bibitem{hunter2018}
\bysame, \emph{The computation of the mean first passage times for {M}arkov
  chains}, Linear Algebra and its Applications \textbf{549} (2018), 100 -- 122.

\bibitem{Janssen:2}
Jacques Janssen and Raimondo Manca, \emph{Semi-{M}arkov risk models for
  finance, insurance, and reliability}, Springer, New York, 2007.

\bibitem{kemeny_snell}
J.G. Kem{\'e}ny and J.L. Snell, \emph{Finite {M}arkov chains}, University
  series in undergraduate mathematics, Van Nostrand, 1960.

\bibitem{khar_solo_uluk10}
Jeffrey~P. Kharoufeh, Christopher~J. Solo, and M.Y. Ulukus, \emph{Semi-{M}arkov
  models for degradation-based reliability}, IIE Transactions \textbf{42}
  (2010), no.~8, 599--612.

\bibitem{kincaid2002numerical}
David Kincaid and Ward Cheney, \emph{Numerical analysis: mathematics of
  scientific computing}, 3rd ed., vol.~2, American Mathematical Society,
  Providence, 2002.

\bibitem{Kleinrock:1}
Leonard Kleinrock, \emph{Queueing systems volume {I}: Theory}, John Wiley \&
  Sons, New York, 1975.

\bibitem{kolpakov1983matrix}
V.V. Kolpakov, \emph{Matrix seminorms and related inequalities}, Journal of
  Mathematical Sciences \textbf{23} (1983), no.~1, 2094--2106.

\bibitem{lerman79}
Steven~R Lerman, \emph{The use of disaggregate choice models in semi-{M}arkov
  process models of trip chaining behavior}, Transportation Science \textbf{13}
  (1979), no.~4, 273--291.

\bibitem{levy_2_54}
P.~Levy, \emph{Processus semi-{M}arkoviens}, Proc. Intern. Congr. Math.
  \textbf{3} (1954), 416--426, Amsterdam, The Netherlands.

\bibitem{levy_1_54}
\bysame, \emph{Systems semi-{M}arkoviens ayant au plus une inifinite
  denombrable d'etats possibles}, Proc. Intern. Congr. Math. \textbf{2} (1954),
  294--295, Amsterdam, The Netherlands.

\bibitem{SemiMarkov}
N.~Limnios and G.~Opri\c{s}an, \emph{Semi-{M}arkov processes and reliability},
  Birkh\"{a}user, Boston, 2001.

\bibitem{meyer2000matrix}
Carl~D. Meyer, \emph{Matrix analysis and applied linear algebra}, Society for
  Industrial Mathematics, Philadelphia, 2000.

\bibitem{neuts_mg1_89}
M.F. Neuts, \emph{Structured stochastic matrices of {$M/G/1$} type and their
  applications}, Probability: Pure and Applied, Marcel Dekker, Inc., New York
  and Basel, 1989.

\bibitem{pyke1}
Ronald Pyke, \emph{{M}arkov renewal processes: Definitions and preliminary
  properties}, The Annals of Mathematical Statistics \textbf{32} (1961), no.~4,
  1231--1242.

\bibitem{pyke2}
\bysame, \emph{{M}arkov renewal processes with finitely many states}, The
  Annals of Mathematical Statistics \textbf{32} (1961), no.~4, 1243--1259.

\bibitem{RossOpt}
Sheldon~M. Ross, \emph{Applied probability models with optimization
  applications}, Holden-Day, San Francisco, 1970.

\bibitem{shao_85}
Jia-Yu Shao, \emph{Products of irreducible matrices}, Linear Algebra and Its
  Applications \textbf{68} (1985), 131--143.

\bibitem{smith_55}
W.~L. Smith, \emph{Regenerative stochastic processes}, Proc. Roy. Soc (GB),
  series A, \textbf{232} (1955), 6--31, Amsterdam, The Netherlands.

\bibitem{veeramany2011reliability}
Arun Veeramany and Mahesh~D Pandey, \emph{Reliability analysis of nuclear
  piping system using semi-markov process model}, Annals of Nuclear Energy
  \textbf{38} (2011), no.~5, 1133--1139.

\bibitem{weiss1965semi}
George~H Weiss and Marvin Zelen, \emph{{A semi-Markov model for clinical
  trials}}, Journal of Applied Probability \textbf{2} (1965), no.~2, 269--285.

\bibitem{MPFirstPass}
David~D. Yao, \emph{First-passage-time moments of {M}arkov processes}, Journal
  of Applied Probability \textbf{22} (1985), no.~4, 939--945.

\bibitem{younes_simm04}
H.L.S. Younes and R.G. Simmons, \emph{Solving generalized semi-{M}arkov
  decision processes using continuous phase-type distributions}, Proceeding of
  the National Conference on Artificial Intelligence, 2004, pp.~742--747.

\bibitem{SMPFirstPass}
Xuan Zhang and Zhenting Hou, \emph{The first-passage times of phase
  semi-{M}arkov processes}, Statistics \& Probability Letters \textbf{82}
  (2012), no.~1, 40 -- 48.

\end{thebibliography}



\vspace{.2in} 

\noi 
{\bf Richard L. Warr} \\
Brigham Young University\\
223 TMCB\\
Provo, UT 84602\\
E-mail: \href{mailto:richard.L.warr@gmail.com}{richard.L.warr@gmail.com} \\

\vspace{.1in}
  
\noi 
{\bf James D. Cordeiro}\\  
University of Dayton\\
300 College Park\\
Dayton, OH 45469\\
E-mail: \href{mailto:jcordeiro1@udayton.edu}{jcordeiro1@udayton.edu}  \\

\end{document}